\documentclass[11pt]{article}

\usepackage{amsmath}
\usepackage{amsthm,amscd,amsfonts}
\usepackage{amssymb, upref, color}
\usepackage{amsmath,amssymb,amsthm,mathrsfs}
\usepackage{color}
\usepackage[dvips]{graphicx}
\usepackage{epsf,epsfig, verbatim} 
\usepackage{latexsym,bm}
\usepackage{indentfirst}
\usepackage{color}
\usepackage{cite}

\usepackage{epsfig}
\usepackage{epstopdf}
\usepackage{subfigure}
\usepackage[font=scriptsize]{caption}
\usepackage{graphicx}
\graphicspath{{figures/}}

\numberwithin{equation}{section}

\theoremstyle{plain}
\newtheorem{exam}{Example}[section]
\newtheorem{theorem}[exam]{Theorem}
\newtheorem{lemma}[exam]{Lemma}
\newtheorem{remark}[exam]{Remark}

\def\bc{\begin{center}}
\def\ec{\end{center}}
\def\be{\begin{equation}}
\def\ee{\end{equation}}
\def\ba{\begin{array}}
\def\ea{\end{array}}
\def\bea{\begin{eqnarray}}
\def\eea{\end{eqnarray}}
\def\beaa{\begin{eqnarray*}}
\def\eeaa{\end{eqnarray*}}

\def\al{\alpha}

\def\la{\lambda}

\def\e{\varepsilon}

\def\oo{\infty}

\def\mcl{{\mathcal L}}

\def\Lp{{\mathcal L}^p}

\def\z{\left}
\def\y{\right}
\def\lb{\label}
\def\x#1{(\ref{#1})}

\begin{document}
	\sloppy
	\captionsetup[figure]{labelfont={bf},name={Fig.},labelsep=period}
	\captionsetup[table]{labelfont={bf},name={Table},labelsep=period}

\title{A novel and application-oriented inverse nodal  problem for Sturm-Liouville operators
\footnote{This paper was jointly supported from NSF-ZJ (No.LZ24A010006) and NNSFC (No.11931016).}	
}
	\author
{
	Yuchao He$^{a}$,\qquad
    Mengda Wu$^{a}$,\qquad
	Yonghui Xia$^{b}$\footnote{Corresponding author. E-mails: yhxia@zjnu.cn, xiadoc@163.com}, \qquad
    Meirong  Zhang$^{c}$\\
	{\small \textit{$^a$ School of Mathematical  Science,  Zhejiang Normal University, 321004, Jinhua, China}}\\
	{\small \textit{$^b$ School of Mathematics, Foshan University, Foshan 528000, China }}\\
	{\small \textit{$^c$ Department of Mathematical Sciences, Tsinghua University, Beijing 100084, China }}\\
	{\small Email:  yuchaohe@zjnu.edu.cn; medawu@zjnu.edu.cn; xiadoc@163.com; zhangmr@tsinghua.edu.cn}
}
	\maketitle
	
	\begin{abstract}
		This paper develops a methodological framework for addressing a novel and application-oriented inverse nodal problem in Sturm-Liouville operators,
 having significant applications in seismic wave analysis and submarine underwater radar (sonar) detection. By utilizing a given finite set of nodal data, we propose an optimization framework to find the potential $\hat q$ that is most closely approximating a predefined target potential $q_0$. 
The inverse nodal optimization problem is reformulated as a solvability problem for a class of nonlinear Schr\"odinger equations, enabling systematic investigation of the inverse nodal problem.
 {As an example, when the constant target potential $q_0$ is considered, we find that the Schr\"odinger equations are completely integrable and conclude that the potential $\hat q$ is `periodic' in a certain sense. Furthermore, the reconstruction of $\hat q$ is reduced to solving a system of three featured parameters, thereby establishing an explicit  quantitative relationship between
 	  $\|\hat q\|_{\Lp}$ and $T_*$. Of importance, we prove the uniqueness of the potential $\hat q$ when $p>3/2$. These new findings represent a substantial advancement in this field of study. Our methodology also bridges theoretical rigor with practical applicability, addressing scenarios where only partial nodal information is available.
 	 }
	\end{abstract}
	

	
	 \section{Introduction}\lb{sec1}
		


Let $\Omega = [0, 1]$ be the unit interval. For a fixed exponent $p \in (1, \infty)$, the $L^p$ Lebesgue space on $\Omega$
is denoted by
    \[
\Lp:=L^p(\Omega,\mathbb R).
    \]
For an integrable potential $q \in \Lp$, we consider the Sturm-Liouville operator
    \begin{equation}\lb{s-l}
\mcl y:=-y''+q(x)y,\quad x\in(0,1)
    \end{equation}
with the Dirichlet boundary conditions
    \begin{equation} \lb{dbc}
	y(0)=y(1)=0.
    \end{equation}
According to the classical spectral theory (e.g. see \cite{book}), problem \x{s-l}-\x{dbc} admits an increasing eigenvalues
    \[
	\lambda_1(q)<\lambda_2(q)\cdots<\lambda_m(q)<\cdots
    \]
and their associated eigenfunctions $E_m(x;q)$, $m=1,2,\cdots$. The interior zeros  of the eigenfunctions $E_m(x;q)$ are called  nodes of problem \x{s-l}-\x{dbc}. Note that $E_1(x;q)$ has no zero in $(0, 1)$ and for $m=2,3,\cdots$, $E_m(x;q)$ has precisely $m-1$ zeros inside $(0,1)$, which are ordered as
    \[
0<T_{1,m}<T_{2,m}<\cdots<T_{m-1,m}<1.
    \]
One can consider $T_{i,m}=T_{i,m}(q)$ as (implicitly defined) nonlinear functionals of $q\in \Lp$.  {For further details, we refer the reader to references \cite{book1,book,book3,book4}.}

{ It is well-known that nodes are critically important and observable physical quantities. For instance, in oscillating systems, nodes correspond to points where the amplitude is zero-positions at which the system does not vibrate.
Therefore, determining the optimal location for nodes $T_{i,m}(q)$ based on given the potential data is a significant topic. Recently,  Guo and Zhang \cite{G-Z2} applied the results of \cite{G-Z} to obtain optimal location of the nodes $T_{1,2}(q)$. Moreover, Chu et al. \cite{jiedian} studied the optimal characterizations of locations for all nodes $T_{i,m}(q)$. Different from the usual studies on estimations of nodes, they have deduced the critical equations, which are used to obtain the optimal locations of nodes.	
The aforementioned studies focus on the direct nodal problem using given potential data. However, in practice, the potential function is usually unknown. Instead, nodal data is typically observable since the system does not vibrate at nodal-positions. In such cases, it becomes necessary to reconstruct the potential function from these known nodal properties.
 The pioneering works on these inverse nodal problems go back to McLaughlin \cite{Mclaughlin}. The main result stated that any {\em dense subset} of the nodal set can uniquely determine the potential $q$ up to a constant. Subsequently, Yang \cite{yang} has established explicit reconstruction formulas for the potential and the boundary data by using {\em a dense subset} of the nodal set. For further results along this line, see, for example, Guo and Wei \cite{guo-wei},  { Wang and Yurko\cite{w-y}, Yang \cite{YANGCF}}. 
}

However, in real-world application, it is impossible to obtain all or a lot of nodes. In contrast to aforementioned works based on the  {\em dense subset} of the nodal set  (McLaughlin \cite{Mclaughlin},  Yang \cite{yang,yang2},  Guo and Wei \cite{guo-wei},  Wang and Yurko\cite{w-y}, Yang \cite{YANGCF}), we propose a novel and application-oriented inverse nodal  problem for  Sturm-Liouville operators based on {\em only finitely known nodal data}. As a starting work, we will study the following problem.  { For a given observed node $T_*$ and a target potential $q_0$, does there exist an optimal potential function $\hat q$ that best approximates the observed potential $q_0$ while satisfying the nodal constraints?
	 Additionally, if the potential $\hat q$ exists, can we recover it by the observed node $T_*$?  To enhance both mathematical rigor and pedagogical clarity, we reformulate the preceding questions as a novel inverse nodal optimization problem, formally defined as follows:}

{\em \noindent {\bf Problem $(INP)$:} Let $T_*\in \mathbb (0,1)$ and $q_0 \in \Lp$ be given. It is required to find a potential  $\hat q \in \Lp$
such that the $(i,m)$-th node $T_{i,m}(\hat{q})$  coincides with the given value $T_*$ and 
    \begin{equation}\label{youhua}
	\|\hat q-q_0\|_{\Lp}=\min\{\|q-q_0\|_{\Lp}:T_{i,m}(q)=T_*,\ q\in \Lp\}.
    \end{equation}
    }    
 {This type of inverse problem is of great practical significance. For example, in seismic wave analysis, $T_*$ denotes the location where the seismic wave amplitude vanishes (zero-crossing point), while $q_0$ represents an initial estimate of the subsurface structure $q$, inferred from auxiliary data such as geological surveys, well-log measurements, regional models, physical constraints, or empirical observations. The formulation in Problem (INP), Equation \eqref{youhua}, integrates $q_0$ as a prior constraint, ensuring that the inversion process yields an optimal solution $\hat{q}$ that not only fits the observed seismic wave data $T_*$ but also best aligns with existing geological knowledge $q_0$.   Another important application of this type of inverse problem is the study of submarine underwater radar (sonar) detection. $T_*$ represents the location where the amplitude of the underwater acoustic wave is zero, and $q_0$ denotes an initial estimate of the true structural parameters $q$ of the underwater object based on prior knowledge (e.g., object shape, density, volume, radius, curvature, acoustic impedance, etc.). In the inverse Problem (INP), Equation \eqref{youhua} incorporates $q_0$ as a prior constraint by introducing object shape and other known information. This guides the inversion process to select, from possible solutions, the optimal solution $\hat q$ that not only fits the observed acoustic wave data $T_*$ but also aligns best with our existing understanding of the detected object $q_0$.
}

This class of inverse nodal problems had never been studied before. However, when nodes $T_{i,m}(q)$ are replaced by eigenvalues $\lambda_m(q)$, similar inverse spectral problems with finitely many data on known eigenvalues had been extensively studied by Valeev and Il'yasov \cite{V-I2019, V-I,V-I3}. Their results contain the critical systems of equations and the uniqueness of the desired potential (for several cases).  For more recent researches on these inverse spectral problems,  {we refer to \cite{H-M,s-s-v,P-S,P-T}.} In this sense, {\bf problem $(INP)$} above can be considered as a counterpart of the problems in these papers. However, as for the existence of the optimal potential $\hat q$ for {\bf problem $(INP)$}, the proof of this paper is in some sense a direct method, which is much simpler and completely different from the previous papers. In fact,  {with the constant target potential $q_0$,} our problems and the approaches we will develop are more or less some inverse to those in the recent paper \cite{jiedian}. Moreover,  for these inverse nodal problems, the critical equations obtained are more concise than in \cite{jiedian, G-Z2}.
 {When the target potential $q_0$ is constant, the critical equations are nonlinear Schr\"odinger equations. We will find three new characteristic features to construct $\hat q$ and to establish a direct and  quantitative relationship between $\|\hat q\|_{\Lp}$ and $T_*$.
Note that for the direct nodal problem, Chu et al. \cite{jiedian} deduced the Schr\"odinger equations and the nonlinear equations characterizing nodal properties, but no rigorous  mathematical proof for uniqueness was provided for the nonlinear equations in \cite{jiedian}. In contrast, for the present inverse problem, we rigorously prove the uniqueness of solutions to these nonlinear equations, thereby establishing the uniqueness of $\hat q$ for the case $p>3/2$. These new results have significantly advanced the study in this direction.

 {
This paper is organized as follows. In next section, firstly,  we prove the  existence of the optimal potential $\hat q$ for {\bf problem $(INP)$} based on the complete continuity of nodes $T_{i,m}(q)$ in $q$ (see \cite{jiedian1} {and \cite{G-Z}}).  
  Secondly, by using the formula for the Fr\'echet derivatives of $T_{i,m}(q)$ in $q$, we  deduce  a critical equation for {\bf problem $(INP)$} {which is a (piecewise) second-order nonlinear Schr\"odinger equation}. In Section 3, we  show that,  if $q_0$ is a constant, {the critical equation is `integrable' and its solution has periodicity from the phase portraits of the critical equation in subsection 3.1. Moreover,   the direct relationship between $\|\hat q\|_{\Lp}$ and $T_*$ is expressed by three new characteristic features in subsection 3.2. This implies that the original infinitely dimensional inverse nodal problem is reduced to solve a system of three nonlinear equations for three parameters. Through three parameters, we present the expression of $\|\hat q\|_{\Lp}$}.
We further prove that the parameters in the three new characteristic nonlinear equations are uniquely determined, which leads to the uniqueness of $\hat q$ when $p>3/2$ in subsection 3.3.
We present a special example ($p=2$)  to further illustrate our results, and demonstrate the relationship between  $\|\hat q\|_{\Lp}$ and $T_*$ through numerical simulations in subsection 3.4.   Finally, a conclusion with an open problem is given.
}

\section{The critical system}
\subsection{Existence of optimal potentials}

Firstly, we introduce the following lemma, which is crucial for the subsequent proofs in this paper.

\begin{lemma} \lb{dd} (see \cite{G-Z})
Given $(i,m)$, $T_{i,m}(q)$ is continuously Fr$\acute{e}$chet differentiable in $q \in(L^1, \|\cdot\|_{L^1})$ and
	\begin{equation}\lb{pT}
	\partial_q T_{i,m}(q)(x)=H_{i,m}(x;q)E_m^2(x;q)\in C^1[0,1],
	\end{equation}
where
	\begin{equation}\lb{H}
		H_{i,m}(x;q)=\begin{cases}
			+a_{i,m}, \ {\rm for}\ x\in[0,T_{i,m}],\\
			-b_{i,m},\ {\rm for} \ x\in (T_{i,m}, 1],
		\end{cases}
	\end{equation}
with the positive constants $a_{i,m}=a_{i,m}(q)$ and $b_{i,m}=b_{i,m}(q)$ being given by
	\[\begin{split}
	a_{i,m}&=\int_{T_{i,m}(q)}^1\left(\frac{E_m(v;q)}{E_m'(T_{i,m}(q);q)}\right)^2dv,
	\\
	b_{i,m}&=\int_0^{T_{i,m}(q)}\left(\frac{E_m(v;q)}{E_m'(T_{i,m}(q);q)}\right)^2dv.
    \end{split}
	\]
\end{lemma}

{Let us notice from formulas \x{pT} and \x{H} that
    \be \lb{ze}
    \int_0^{T_{i,m}(q)} a_{i,m} (q)E_m^2(v;q) dv = \int_{T_{i,m}(q)}^1  b_{i,m}(q)E_m^2(v;q) dv.
    \ee
 Formula \x{ze} implies that $\int_0^1 H_{i,m}(v;q)E_m^2(v;q)dv =0$, which is  coincident with the translation invariance of nodes
    \[
T_{i,m}(q+c)=T_{i,m}(q)\qquad \forall c\in {\mathbb R}.
    \]
Namely, in the inverse problems using nodes, there may be multiple different potentials which are corresponding to the same observation values $T_{i,m}$. The uniqueness can be obtained only up to a constant.
}

    \begin{theorem}\lb{existence}
Let $q_0 \in \Lp$ be a given potential. Then there exists a potential $\hat{q}\in \Lp$ which solves {\bf problem $(INP)$}.
	\end{theorem}

	\begin{proof}
Let
    $$
    S_{T_*}=\{q\in \Lp:T_{i,m}(q)=T_*\}.
    $$
Since $\partial_q T_{i,m}(q)\ne 0$ (see \x{pT}), one knows that $S_{T_*}$ is a 1-codimension differentiable sub-manifold of $(\Lp, \|\cdot\|_{\Lp})$. For {\bf problem $(INP)$}, there exists a minimizing potential sequence $\{q_n\}$ such that $q_n\in S_{T_*}$ $(n\in \mathbb N)$ and
    \begin{equation}\lb{1.5}
			\|q_n-q_0\|_{\Lp}\to \inf_{q\in S_{T_*}}\|q-q_0\|_{\Lp}<+\infty.
    \end{equation}
This implies that $\|q_n\|_{\Lp}$ $(n\in \mathbb N)$ are bounded. Without loss of generality, let us assume that $q_n\rightharpoonup \hat q$ in $\Lp$. According to the property of weak convergence, from \eqref{1.5}, we can obtain
		\begin{equation}\lb{1.6}
			\|\hat q-q_0\|_{\Lp}\leq\liminf\limits_{n\to\infty}\|q_n-q_0\|_{\Lp}=\inf_{q\in S_{T_*}}\|q-q_0\|_{\Lp}.
		\end{equation}

Since nodes are completely continuous in potentials \cite[Theorem 1.2]{G-Z}, we have
		\[
		T_{i,m}(\hat q)=\lim_{n\to\infty}T_{i,m}(q_n)=T_*,
		\]
i.e. $\hat q\in S_{T_*}.$ From  inequality  \eqref{1.6}, it follows that
		\be \lb{RT}
		\|\hat q-q_0\|_{\Lp}=\inf_{q\in S_{T_*}}\|q-q_0\|_{\Lp}=\min_{q\in S_{T_*}}\|q-q_0\|_{\Lp}=: R_{T_*}.
		\ee
This completes the proof.
	\end{proof}


\subsection{The critical equation for optimal potentials}\lb{sec21}

Let us assume in {\bf problem $(INP)$} that $T_{i,m}(q_0)\ne T_*$. Therefore $\hat q\ne q_0$. Otherwise, the trivial solution to {\bf problem $(INP)$} is simply $\hat q=q_0$.

In this subsection, we apply the Lagrange multiplier method to transform the inverse nodal optimization problems into the solvability of a class of Schr\"odinger equations. Due to this, some systematic investigations of the inverse nodal problems can be studied using the usual approach. {However, we have not addressed the uniqueness of the minimal potentials $\hat q$.}

Due to Theorem \ref{existence}, {\bf problem $(INP)$} is a minimization problem on the aim functional $\|q-q_0\|_{\Lp}$ with a constraint
    \[
    T_{i,m}(q)=T_*.
    \]
For an exponent $p \in (1, \infty)$, the increasing homeomorphism $\phi_{p}(s) : \mathbb R \to \mathbb R$ is defined by
    \begin{equation}\lb{phip}
	\phi_p(s)=|s|^{p-2}s, \quad \rm{for}\  s\in \mathbb R,
    \end{equation}
and $\phi_p^{-1}(t)=|t|^{p^*-2}t$, where $p^*=\frac{p}{p-1}\in (1,\infty)$ is the conjugate exponent of $p$. Note that the Fr\'echet derivatives $\partial_q T_{i,m}(q)$ are given by  \x{pT} and \x{H}. Moreover, one has
    \be \lb{pq}
    \partial_q \|q-q_0\|_{\Lp}(x)=\|q-q_0\|_{\Lp}^{1-p} \phi_p(q(x)-q_0(x)), \quad x\in\Omega,
    \ee
for $q\in \Lp \setminus \{q_0\}$. Applying the Lagrange multiplier method to a minimizing potential $\hat q$ of {\bf problem $(INP)$}, we obtain
    \begin{equation}\lb{qq0}
	\|\hat q-q_0\|_{\Lp}^{1-p} \phi_p(\hat q(x)-q_0(x))=c H_{i,m}(x;\hat q)(E_{i,m}(x;\hat q))^2,\quad x\in{\Omega},
    \end{equation}
where the constant $c=c_{\hat q}\ne0$.

By using \x{phip}, equation \x{qq0} reads as
    \begin{equation}\lb{eq1}
\phi_p(\hat q(x)-q_0(x))=\begin{cases}
	+\tilde a_{i,m}(E_{m}(x;\hat q))^2,	\quad x\in{[0,T_*]},\\
		-\tilde b_{i,m}(E_{m}(x;\hat q))^2,	\quad x\in{(T_*,1]},
\end{cases}
	\end{equation}
where $T_*$ is as in {\bf problem $(INP)$} and the constants $\hat a_{i,m}$ and $\hat b_{i,m}$  are
    \[
    \hat a_{i,m}= \|\hat q-q_0\|_{\Lp}^{p-1} a_{i,m}(\hat q) c , \quad \hat b_{i,m}= \|\hat q-q_0\|_{\Lp}^{p-1} b_{i,m}(\hat q) c.
    \]	
For simplicity, denote the sign
    \[
    \e:={\rm sign}(c) = \pm 1.
    \]
Then
    \[
    {\rm sign}(\hat a_{i,m})={\rm sign}(\hat b_{i,m})= {\rm sign}(c)=\e.
    \]

By \x{eq1}, let us introduce the function
    \begin{equation}\lb{umx}
	u_m(x):=\begin{cases}
		\sqrt{|\hat a_{i,m}|}\, E_m(x;\hat q),\quad x\in[0,T_*],\\
        \sqrt{|\hat b_{i,m}|}\, E_m(x;\hat q),\quad x\in(T_*,1].
	\end{cases}
	\end{equation}
{This is a piecewise scaling `eigenfunction' of $E_m(x;\hat q)$, because $u_m(x)$ is no longer differentiable at $x=T_*$.}
Then we can write the critical equation \x{eq1} as
    \begin{equation}\lb{eq2}
\hat q(x)-q_0(x) =\begin{cases}+\e \phi_{p^*}(u_m^2(x))= +\e |u_m(x)|^{2p^*-2},\quad x\in[0,T_*],\\
	-\e\phi_{p^*}(u_m^2(x))= -\e|u_m(x)|^{2p^*-2},\quad x\in(T_*,1].
\end{cases}
	\end{equation}{An important observation on $u_m(x)$ which is deduced from \x{ze} is the following equality
    \be \lb{ze1}
    \int_0^{T_*} u_m^2(x) dx = \int_{T_*}^1  u_m^2(x) dx.
    \ee
Moreover, due to \x{RT}, we have 
    \be \lb{RT1}
    R_{T_*}=\z\| |u_m|^{2p^*-2}\y\|_{\Lp} = \left(\|u_m\|_{\mathcal L^{2p^*}}\right)^{2/(p-1)}.
    \ee}

{Since $u_m(x)$ can change signs on both of $[0,T_*]$ and $(T_*,1]$, the critical equation has the following different forms on the intervals $[0, T_*]$ and $[T_*,1]$.}
	Let us notice from the first line of \x{umx} and from the equation for eigenfunction $E_m(x;\hat q)$ that
	\[
	-u''_m + \hat q(x) u_m=\la_m u_m,\quad x\in[0,T_*].
	\]
	By using the first case of \x{eq2},  {this equation is transformed into the following form:}
	\[
		-u_m''+q_0(x)u_m+\e \phi_{2p^*}(u_m)=\lambda_m u_m,\quad  x\in[0,T_*],
	\]
	 because
	\[
	|u_m(x)|^{2p^*-2}\cdot u_m(x)\equiv \phi_{2p^*}(u_m(x)).
	\]
	 {For the case where 
	$x\in[T_*,1]$, repeating the above derivation  and incorporating the Dirichlet boundary conditions \eqref{dbc}, we immediately obtain the following critical equation.
	 \begin{equation}\lb{P}
		\begin{cases}
			-u_m''+q_0(x)u_m+\e \phi_{2p^*}(u_m)=\lambda_m u_m,\quad  x\in[0,T_*],\\
			-u_m''+q_0(x)u_m-\e \phi_{2p^*}(u_m)=\lambda_m u_m,\quad  x\in[T_*,1],
		\end{cases}
	\end{equation}
	{where $\la_m:= \la_m(\hat q)$.}
	Moreover, $u_m(x)$ satisfies
	\begin{equation} \lb{bc1}
		u_m(0)=u_m(T_*)=u_m(1)=0,
	\end{equation}
	and $u_m(x)$ has precisely $(i-1)$ zeros in $(0,T_*)$, and $(m-1-i)$ zeros in $(T_*,1)$. At this stage, we have transformed the original optimization problem into a Schr\"odinger equation problem and present the following theorem.
    \begin{theorem}
For {\bf problem $(INP)$}, the piecewise scaling `eigenfunction' $u_m(x)$  satisfies the critical equation
   \eqref{P}.
    \end{theorem}
}

		\begin{remark}
The critical equation derived in this paper differs slightly from that in \cite{jiedian}. This difference arises from our adoption of a new normalization method, which removes the influence of the function $H(x)$ on the critical equation, thereby simplifying it. Interestingly, unlike direct nodal problems, the sign of $\e$ in the critical equation of the inverse nodal problem is not fixed and may vary. A detailed analysis of the sign of $\e$ is provided in the next section.
			\end{remark}

	\section{The reconstruction of minimal potentials}\lb{sec4}

	In this section, we  focus on the case when $q_0$ is a constant and find the optimal potential by solving the critical  equation. A necessary result is as follows.

    \begin{lemma}\lb{lam}
		If $q_0$ is a constant, then the $m$-$th$ eigenvalue  {in \eqref{P}}  satisfies $\lambda_m=\la_m(\hat q)>q_0$.
    \end{lemma}

		\begin{proof}	
By integrating equations \eqref{P} on $[0,T_*]$ and on $[T_*,1]$ respectively and using conditions in \x{bc1}, we can obtain
	\begin{equation}\lb{jifen}
		\begin{cases}
		\int_0^{T_*}(u_m')^2dx+\int_0^{T_*}q_0(x)u_m^2dx+\frac{\e}{p^*}\int_0^{T_*}|u_m|^{2p^*}dx=\lambda_m\int_0^{T_*} u_m^2dx,\\
		\int_{T_*}^1(u_m')^2dx+\int_{T_*}^1q_0(x)u_m^2dx-\frac{\e}{p^*}\int_{T_*}^1|u_m|^{2p^*}dx=\lambda_m\int_{T_*}^1 u_m^2dx.
		\end{cases}
	\end{equation}
If $\e=+1$, according to the first equation of \eqref{jifen}, we obtain that $\lambda_m>q_0$. If $\e=-1$, the conclusion is obtained through the second equation of \eqref{jifen}.
		\end{proof}

\subsection{The first integrals and phase portraits}

Since $q_0$ is a constant, without loss of generality, one can take $q_0=0$. By Lemma \ref{lam}, one has $\la_m>0$. Since the relation between $\hat q(x)$ and $u_m(x)$ are clearly stated in \x{eq2}, we concentrate the study on $u_m(x)$, which solves equations
    \begin{equation}\lb{P1}
	\begin{cases}
	-u_m''+\e \phi_{2p^*}(u_m)=\lambda_m u_m,\quad  x\in[0,T_*],\\
	-u_m''-\e \phi_{2p^*}(u_m)=\lambda_m u_m,\quad  x\in[T_*,1].
	\end{cases}
    \end{equation}

   {
    According to the qualitative theory of planar ordinary differential systems, it is not difficult to conclude that the system \eqref{P1} exhibits a certain periodicity over interval $[0,T_*]$ and $[T_*,1]$.  For the case of $\e=+1$, the phase portraits of the equations \eqref{P1} are shown in Fig. \ref{fig:img1} and Fig. \ref{fig:img2}. For $\e=-1$, the phase portraits of equations \eqref{P1} over
    $x\in[0,T_*]$ and
    $x\in[T_*,1]$  correspond to the phase portraits of equations \eqref{P1} over
     $x\in[T_*,1]$ and
    $x\in[0,T_*]$ respectively, for the case $\e=+1$.
}

\begin{figure}[htbp]
	\centering
	\begin{minipage}{0.48\textwidth}
		\includegraphics[width=\textwidth, height=6cm]{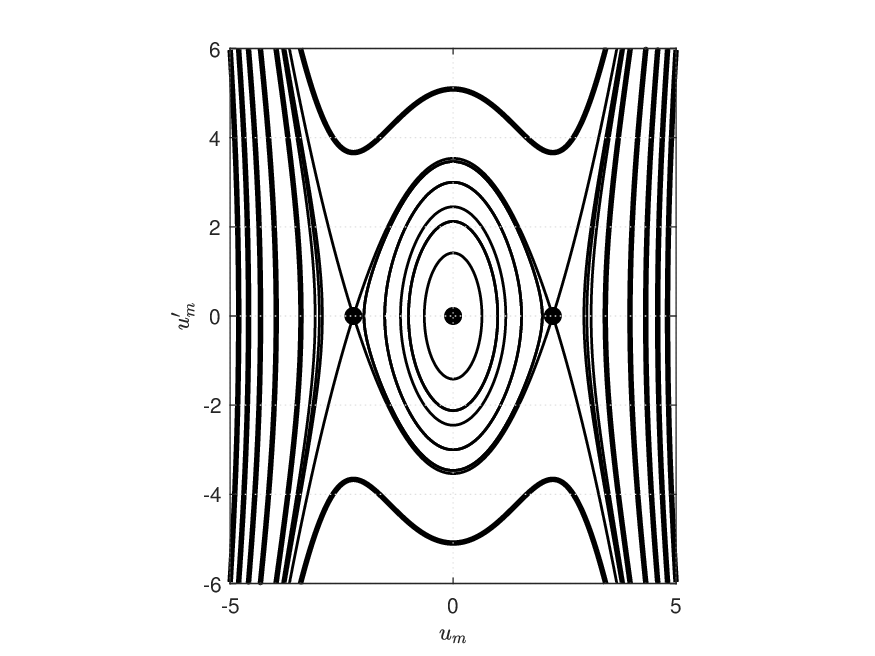}
		\caption{Phase portrait of the first  critical equation ($x\in[0, T_*]$) of \eqref{P1} for $\epsilon=+1$.}
		\label{fig:img1}
	\end{minipage}
	\hfill
	\begin{minipage}{0.48\textwidth}
		\includegraphics[width=\textwidth,height=6cm]{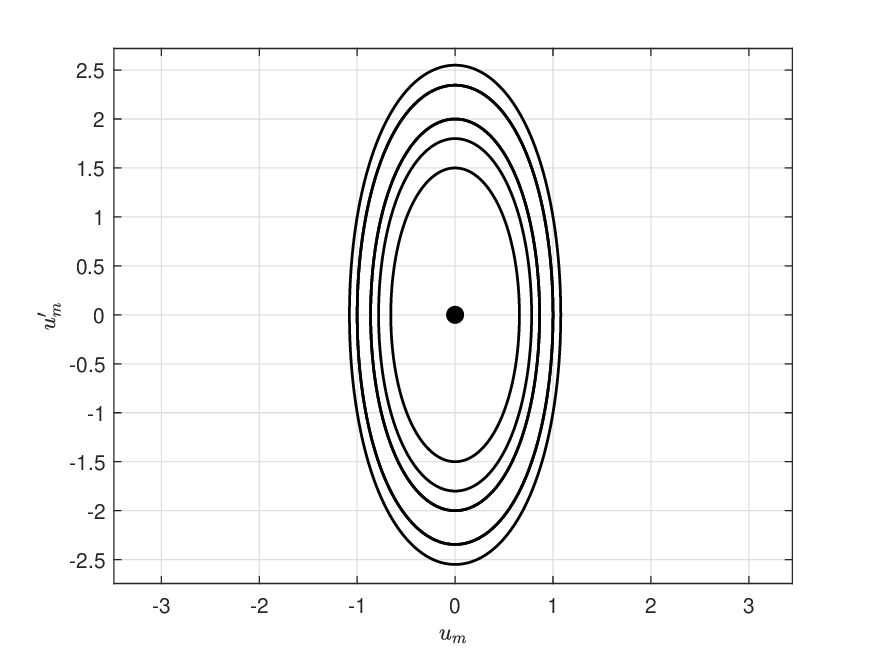}
		\caption{Phase portrait of the second critical equation  ($x\in[ T_*, 1]$) of \eqref{P1} for $\epsilon=+1$.}
		\label{fig:img2}
	\end{minipage}
\end{figure}

 Note that the first integrals of equation \x{P1} are
		\begin{equation}\lb{shouci1}
		\begin{cases}
		(u_m'(x))^2-\frac{\e}{p^*}|u_m(x)|^{2p^*}+\la_m u_m^2(x)=k,\quad x\in [0,T_*],\\
		(u_m'(x))^2+\frac{\e}{p^*}|u_m(x)|^{2p^*}+\la_m u_m^2(x)=\tilde k,\quad  x\in[T_*,1].
		\end{cases}
		\end{equation}
Here $k$ and $\tilde k$ are two (different) constants. It is important to notice from \x{shouci1} that the phase portraits are symmetric with respect to $u_m$ and to $u'_m$. Hence the zeros (nodes) of $u_m(x)$ are evenly distributed on $[0,T_*]$ and on $[T_*, 1]$ respectively.

\subsection{The direct relationship between $T_*$ and $\|\hat q\|_{\Lp}$}

In the preceding section, we have demonstrated that the inverse nodal problem can be transformed into solving critical equations (a class of Schr\"odinger equations). Here, we further elucidate that the sign of $\e$ in the critical equations correlates with $T_*$, which is precisely governed by the following relationship.






\begin{theorem}\label{ep_sign}
 {The value of $\e$ is given as:} 
\begin{equation}\label{ep}
\e=
\begin{cases}
	-1,&\quad {\rm if}\  0<T_*<\frac{i}{m},\\
	+1,& \quad {\rm if}\  \frac{i}{m}< T_*<1.
\end{cases}
\end{equation}	
\end{theorem}

	\begin{proof}
		{
We now consider the first equation of \eqref{shouci1}. Since $\la_m>0$,  the solution $u_m(x)$ is a $P_1$-periodic function.   {The parameters ${h_m}$ and $\alpha_m$ are given by
	 \[
	 {h_m}:=\max_{x\in[0,T_*]}|u_m(x)|,\quad
	 \alpha_m:=\frac{h_m^{2p^*-2}}{\la_m},
	 \]
	 and they can be  determined by each other as follows
	 \begin{equation}\lb{4.3}
	 	\begin{cases}
	 		{h_m}=(\la_m\alpha_m)^{\frac{p-1}{2}},\\
-\e	 	\frac{p-1}{p}{h_m}^{\frac{2p}{p-1}}+\la_m{h_m}^2=k,\\
	 	-\e	\frac{p-1}{p}\alpha_m^p+\alpha_m^{p-1}=\frac{k}{\la_m^p}.
	 	\end{cases}
	 \end{equation}
}	
It should be noted that for different value of $\e$, the range of values for parameters $\alpha_m$ and $h_m$ are different.

If $\e=-1$, it is not difficult to see that $\alpha_m,\ h_m\in\mathbb R^+=(0,+\oo)$.

If $\e=+1$, $\la_m$ must satisfy the condition $\la_m^p>pk$ to ensure that
	\[
	\max_{h\in\mathbb R^+}\left(\la_m h^2-\frac{p-1}{p}h^{\frac{2p}{p-1}}\right) =\left(\la_mh_m^2-\frac{p-1}{p}h_m^{\frac{2p}{p-1}}\right)\bigg|_{h_m={\la_m^{\frac{p-1}{2}}}}=\frac{\la_m^p}{p}>k.
	\]
	Moreover, we deduce that
	$h_m\in(0,\la_m^{\frac{p-1}{2}})$.  It follows form the relationship between $h_m$ and $\al_m$ that  $\al_m\in(0,1)$.	
	 Combining \eqref{4.3} with \eqref{shouci1}, we obtain that
	 \begin{equation}
	 	\frac{du_m}{dx}=\sqrt{\lambda_m}\sqrt{{h_m}^2-u_m^2-\e\frac{p-1}{p\la_m}({h_m}^{2p^*}-u_m^{2p^*})},\quad x\in\left[0,{P_1}/
{4}\right].
	 \end{equation}
It follows form the periodicity of $u_m$ that $u_m(0)=0$ and $u_m(P_1/4)={h_m}$.
 Therefore,
	 \[
	 \frac{P_1}{4}=\int_0^{P_1/4}dx=\int_0^{h_m} \frac{du_m}{u_m'}.
	 \]
By setting $u_m={h_m} t$, we have
	 \[
P_1=\frac{4}{\sqrt{\la_m}}\int_0^1\frac{dt}{\sqrt{1-t^2-\e\frac{\alpha_m}{p^*}(1-t^{2p^*})}}.	
	 \]
{Moreover, we define
    \[
	 \underline{{\bf T}}_{p}(\alpha)\equiv \int_0^1\frac{dt}{\sqrt{1-t^2-\e\frac{\alpha}{p^*}(1-t^{2p^*})}},\ \alpha\in(0,+\infty).
	 \]
}
Then
    \[
    P_1=\frac{4}{\sqrt{\la_m}}\underline{{\bf T}}_{p}(\alpha_m).
    \]
Noting that the gap between the neighboring zeros is $\frac{2\underline{{\bf T}}_{p}(\alpha_m)}{\sqrt{\la_m}}$ and $u_m(T_*)=0$, we obtain that
	\begin{equation}\lb{0+T}
	T_*=\frac{2i\underline{{\bf T}}_{p}(\alpha_m)}{\sqrt{\la_m}}.
    \end{equation}

Next, we consider the second equation of \eqref{shouci1}. Set
    \begin{equation}
	{\widetilde{h}}_m\equiv\max_{x\in[0,T]}|u_m(x)|,\ \beta_m\equiv\frac{\widetilde{h}_m^{2p^*-2}}{\la_m}.
    \end{equation}
Similar to the previous discussion, we have
    \[
    {\widetilde{h}}_m\in \mathbb R^+,\ \beta_m\in \mathbb R^+,\quad {\rm for}\ \e=+1
    \]
and
    \[
    {\widetilde{h}}_m\in(0,\la^{p-1}),\ \beta_m\in (0,1),\quad {\rm for}\ \e=-1.
    \]
According to \eqref{shouci1}, ${\widetilde{h}_m}$ is determined by
    \begin{equation}\lb{4.12}
\la_m{\widetilde{h}_m}^2+\e\frac{1}{p^*}{\widetilde{h}_m}^{2p^*}=\tilde{k}.
    \end{equation}
Since
    \[
\frac{du_m}{dx}=\frac{1}{\sqrt{\la_m}\sqrt{{\widetilde{h}_m}^2-u_m^2+\e\frac{1}{p^*\la_m}({\widetilde{h}_m}^{2p^*}-u_m^{2p^*})}},
    \]
we have
    \[
\frac{P_2}{4}=\int_0^{P_2}dx=\int_0^{\widetilde{h}_m}\frac{1}{u_m'}du_m.
    \]
Set $u_m={\widetilde{h}_m} t$, then we obtain
    \[
P_2=\frac{4}{\sqrt{\la_m}}\int_0^1\frac{1}{\sqrt{1-t^2+\e\frac{\beta_m(1-t^{2p^*})}{p^*}}}dt.
    \]
{Moreover, we define
    \begin{equation}
	\overline{{\bf T}}_{p}(\beta)\equiv \int_0^1\frac{1}{\sqrt{1-t^2+\e\frac{\beta(1-t^{2p^*})}{p^*}}}dt,\quad \beta\in(0,1),
    \end{equation}
    }
which implies that the period of $u_m(x)$ is \begin{equation} P_2=\frac{4\overline{{\bf T}}_{p}(\beta_m)}{\sqrt{\la_m}}.\end{equation}
According to the information of the node, we obtain that
    \begin{equation}\lb{1-T}
1-T_*=\frac{2(m-i)\overline{{\bf T}}_{p}(\beta_m)}{\sqrt{\la_m}}.
    \end{equation}
It follows from  \eqref{0+T}, \eqref{1-T} and Theorem \ref{P} that the periods of the potential are
    \[
P_1=\frac{2T_*}{i},\quad x\in[0,T_*]
    \]
and
    \[
P_2=\frac{2(1-T_*)}{m-i},\quad x\in[T_*,1].
    \]
   For the case  $ 0<T_*<\frac{i}{m}$, note that 
   \[
   \frac{T_*}{2i}<\frac{1}{2m}<\frac{1-T_*}{2(m-i)}.
   \]
   Based on \eqref{0+T} and \eqref{1-T}, we have
   \begin{equation}
   	\begin{cases}
   		\frac{T_*}{2i}=\frac{\underline{{\bf T}}_{p}(\alpha_m)}{\sqrt{\la_m}},\\
   		\frac{1-T_*}{2(m-i)}=\frac{\overline{{\bf T}}_{p}(\beta_m)}{\sqrt{\la_m}}.
   	\end{cases}
   \end{equation}
   Moreover, we obtain that
   \[
   \frac{\underline{{\bf T}}_{p}(\alpha_m)}{\sqrt{\la_m}}<\frac{\overline{{\bf T}}_{p}(\beta_m)}{\sqrt{\la_m}}
   \]
   and
   \[\e<0.\]
  For the case of $\frac{i}{m}< T_*<1$, we  similarly obtain that $\e>0.$
}
This directly yields \eqref{ep}.
    \end{proof}

\begin{remark}
	 For the case
	$T_*=\frac{i}{m}$, it is not different to  concluded that
		$\hat q=0.$ Consequently, from equation \eqref{eq2},  { $\e$ can be $+1$ or $-1$.}
\end{remark}

\begin{remark}
	In the direct nodal problem,  {Chu et al.} \cite{jiedian} derived a critical equation to determine the minimum node. The coefficient of the $u_m^{2p^*}$ term in the critical equation maintains a fixed sign. Different from the direct nodal problems, interestingly, Theorem \ref{ep_sign} here reveals that the sign $\e$ as coefficient of the $u_m^{2p^*}$ in the critical equation of the inverse nodal problem is not fixed and may vary in the different interval associated with the value of $T_*$ (see \eqref{ep}).
\end{remark}

By considering the symmetry, it follows from equality \x{ze1} that
\be \lb{ze2}
2i\int_0^{T_*/2i} u_m^2(x)dx=\int_0^{T_*} u_m^2(x) dx = \int_{T_*}^1  u_m^2(x) dx= 2(m-i) \int_{1-T_*/(2m-2i)}^1 u_m^2(x)dx.
\ee
Moreover, based on variable substitution, we obtain that \[\begin{split}
\int_0^{\frac{T_*}{2i}}u_m^2(x)dx&=\int_0^{h_m}\frac{u_m^2(x)}{u'_m(x)}du_m(x)\\
&=\la_m^{p-\frac{3}{2}}\alpha_m^{p-1}\int_0^1\frac{t^2}{\sqrt{1-t^2-\e\frac{\alpha_m(1-t^{2p^*})}{p^*}}}dt
\end{split}
\]
and
\[
\begin{split}
	\int_{1-\frac{T_*}{2m-2i}}^1u_m^2(x)dx&=	\int_{\widetilde{h}_m}^0\frac{u_m^2(x)}{u'_m(x)}du_m(x)\\
	&=\la_m^{p-\frac{3}{2}}\beta_m^{p-1}\int_0^1\frac{t^2dt}{\sqrt{1-t^2+\e\frac{\beta_m(1-t^{2p^*})}{p^*}}}.
\end{split}
\]
   {
    Let
    \begin{equation}\label{uv}
    \underline{{\bf V}}_{p}(\al)=\alpha^{p-1}\int_0^1\frac{t^2}{\sqrt{1-t^2-\e\frac{\alpha(1-t^{2p^*})}{p^*}}}dt
    \end{equation}
    and
    \begin{equation}\label{ov}
    \overline{{\bf V}}_{p}(\beta)=\beta^{p-1}\int_0^1\frac{t^2}{\sqrt{1-t^2+\e\frac{\beta(1-t^{2p^*})}{p^*}}}dt.
    \end{equation}
    }
According to \eqref{ze2}, we obtain that
\begin{equation}\lb{third eq}
	i\underline{{\bf V}}_{p}(\alpha_m)=(m-i)\overline{{\bf V}}_{p}(\beta_m).
\end{equation}
Since  $\underline{{\bf T}}_{p}(\alpha_m)$  is a monotonic function with respect to $\alpha_m$, 
 we have
\begin{equation}\lb{al}
\al_m=\underline{{\bf T}}_{p}^{-1}\left(\frac{\sqrt{\la_m}T_*}{2i}\right).
\end{equation}
Similarly,
\begin{equation}\lb{beta}
\beta_m=\overline{{\bf T}}_{p}^{-1}\left(\frac{ \sqrt{\la_m}(1-T_*)}{2(m-i)}\right).
\end{equation}
According to \eqref{third eq}, \eqref{al} and \eqref{beta}, $\la_m$ turns to be the solution of the following equation
\begin{equation}\lb{lambda}
	i\underline{{\bf V}}_{p}\left(\underline{{\bf T}}_{p}^{-1}\left(\frac{\sqrt{\la_m}T_*}{2i}\right)\right)=(m-i)\overline{{\bf V}}_{p}\left(\overline{{\bf T}}_{p}^{-1}\left(\frac{ \sqrt{\la_m}(1-T_*)}{2(m-i)}\right)\right).
	\end{equation}

The unknown  variables $\lambda_m(T_*),\ \alpha_m(T_*),\ \beta_m(T_*)$ can theoretically be solved from equations \eqref{al}, \eqref{beta} and \eqref{lambda}.

Note that
 \[\begin{split}
(\|\hat q\|_{p})^p&=\int_0^1|q(x)|^{p}dx\\
&=\int_0^{T_*}|u_m(x)|^{2p^*}dx+\int_{T_*}^1|u_m(x)|^{2p^*}dx\\
&=2i\int_0^{h_m}\frac{|u_m|^{2p^*}du_m}{\sqrt{\lambda_m}\sqrt{h_m^2-u_m^2-\e\frac{h_m^{2p^*}-u_m^{2p^*}}{p^*\lambda_m}}}+2(m-i)\int_0^{\widetilde{h}_m}\frac{|u_m|^{2p^*}du_m}{\sqrt{\lambda_m}\sqrt{\widetilde{h}_m^2-u_m^2+\e\frac{\widetilde{h}_m^{2p^*}-u_m^{2p^*}}{p^*\lambda_m}}}.
\end{split}
\]
By taking $u_m(x)=h_mt$ for $x\in[0,T_*]$ and $u_m(x)=\widetilde{h}_mt$ for $x\in[T_*,1]$, we obtain that
\begin{equation}\lb{qno}
\begin{split}
\|\hat q\|_{p}=\left(\frac{2i(\lambda_m\alpha_m)^{(p-1)p^*}}{\sqrt{\lambda_m}}\underline {\bf U}_{p}(\alpha_m)+\frac{2(m-i)(\lambda_m\beta_m)^{(p-1)p^*}}{\sqrt{\lambda_m}}\overline {\bf U}_{p}(\beta_m)\right)^{\frac{1}{p}},
\end{split}
\end{equation}
{where
\[
\underline {\bf U}_{p}(\alpha)=\int_0^1\frac{t^{2p^*}dt}{\sqrt{1-t^2-\e\alpha\frac{(1-t^{2p^*})}{p^*}}}
\]
and
\[
\overline {\bf U}_{p}(\beta)\int_0^1\frac{t^{2p^*}dt}{\sqrt{1-t^2+\e\beta\frac{(1-t^{2p^*})}{p^*}}}.
\]
}

 Through the functions 
$\underline T_p,\overline T_p,\underline V_p,\overline V_p,\underline U_p,\overline U_p$, we have recovered ${\|\hat q\|_{\Lp}}$ of the potential function by three characteristic features $\alpha_m,\beta_m, \la_m $ through a system of the following nonlinear equations \eqref{sys}. 
Since variables $\alpha_m(T_*),\ \beta_m(T_*), \ \lambda_m(T_*)$ are all related to $T_*$, we have established a quantitative relationship between $\|\hat q\|_{\Lp}$ and $T_*$ by  \eqref{qno}.
 We formulate this result as the following theorem:
 
 \begin{theorem}
	The quantitative relationship between ${\|\hat q\|_{\Lp}}$  and $T_{*}$ is established  { in  \eqref{qno}} through  the three characteristic features $\alpha_m$, $\beta_m$, $\lambda_m$ satisfying the following system of nonlinear equations
	\begin{equation} \label{sys}
		\begin{cases}
			\al_m=\underline{{\bf T}}_{p}^{-1}\left(\frac{\sqrt{\la_m}T_*}{2i}\right),\\
			\beta_m=\overline{{\bf T}}_{p}^{-1}\left(\frac{ \sqrt{\la_m}(1-T_*)}{2(m-i)}\right),\\
			i\underline{{\bf V}}_{p}(\alpha_m)=(m-i)\overline{{\bf V}}_{p}(\beta_m).
		\end{cases}
	\end{equation}
\end{theorem}

\begin{remark}
		We reduce the inverse nodal problem {$(INP)$}
 to a solvable system of three nonlinear equations \eqref{sys}. Although the explicit expression of the solutions to \eqref{sys} are not obtained, the problem has been simplified {in a significant way}.
	Once the solution  $\lambda_m=\lambda_*$ to the equation \eqref{lambda} is obtained, the expression for   $\|\hat q\|_{\Lp}$ is given by
	\[
	\begin{split}
	\|\hat q\|_{\Lp}=&\left\{\frac{2i\left(\lambda_*\underline{{\bf T}}_{p}^{-1}\left(\frac{\sqrt{\la_*}T_*}{2i}\right)\right)^{(p-1)p^*}}{\sqrt{\lambda_*}}\underline {\bf U}_{p}\left(\underline{{\bf T}}_{p}^{-1}\left(\frac{\sqrt{\la_*}T_*}{2i}\right)\right)\right.\\
	&\left.+\frac{2(m-i)\left(\lambda_*\overline{{\bf T}}_{p}^{-1}\left(\frac{ \sqrt{\la_*}(1-T_*)}{2(m-i)}\right)\right)^{(p-1)p^*}}{\sqrt{\lambda_*}}\overline {\bf U}_{p}\left(\overline{{\bf T}}_{p}^{-1}\left(\frac{ \sqrt{\la_*}(1-T_*)}{2(m-i)}\right)\right)\right\}^{\frac{1}{p}}.
	\end{split}
	\]
	
\end{remark}

\subsection{The uniqueness of the potential $\hat q$}
  {As an  important problem, is the solution to  {\bf problem $(INP)$} unique?  We  prove} that for
$p>\frac{3}{2}$, the uniqueness of $\hat q$ is guaranteed. We  first prove the uniqueness of $\lambda_m$.

{
\begin{theorem} \label{3.2}
	For $p>\frac{3}{2}$, the solution $\la_m$ of equation \eqref{lambda} is unique. Consequently, the minimal potential $\hat q$ which optimizes {\bf problem $(INP)$} is also unique.
	\end{theorem}
}

	\begin{proof}		
	Here, we only prove the case where
{$T_*\in(0,\frac{i}{m})$.
The proof for the case where
$T_*\in(\frac{i}{m},1)$ is analogous.}

		Let
		\[
		G(\lambda_m)=
		i\underline{{\bf V}}_{p}\left(\underline{{\bf T}}_{p}^{-1}\left(\frac{\sqrt{\la_m}T_*}{2i}\right)\right)-(m-i)\overline{{\bf V}}_{p}\left(\overline{{\bf T}}_{p}^{-1}\left(\frac{ \sqrt{\la_m}(1-T_*)}{2(m-i)}\right)\right),\quad {\rm for}\ \lambda_m\in\mathbb R^+.\]
Define
		\[\begin{split}
		D(\alpha_m, t) &= 1 - t^2 + \frac{\alpha_m}{p^*}(1 - t^{2p^*}),
		\\
		\overline D(\beta_m, t) &= 1 - t^2 - \frac{\beta_m}{p^*}(1 - t^{2p^*}).
    \end{split}
		\]
		Note that
		\[
		G'(\lambda_m) = i \frac{d}{d\lambda_m} \underline{{\bf V}}_{p}(\alpha_m) - (m - i) \frac{d}{d\lambda_m} \overline{{\bf V}}_{p}(\beta_m),
		\]
		{where
\[
		\begin{split}
			\frac{d}{d\lambda_m} \underline{{\bf V}}_{p}(\alpha_m)& = \underbrace{\left[(p-1)\alpha_m^{p-2}\int_0^1 \frac{t^2}{\sqrt{D(\alpha_m, t)}}\, dt + \alpha_m^{p-1}\frac{d}{d\alpha_m}\z(\int_0^1 \frac{t^2}{\sqrt{D(\alpha_m, t)}}\, dt\y)\right]}_{\frac{d \underline{{\bf V}}_{p}}{d \alpha_m}} \\
& \cdot \underbrace{\frac{1}{\underline{{\bf T}}_{p}'(\alpha_m)}}_{<0} \cdot \underbrace{\frac{T_*}{4i\sqrt{\lambda_m}}}_{>0}
		\end{split}
\]
		and
\[
		\begin{split}
			\frac{d}{d\lambda_m} \overline{{\bf V}}_{p}(\beta_m)
			&= \underbrace{\left[(p-1)\beta_m^{p-2}\int_0^1 \frac{t^2}{\sqrt{\overline D(\beta_m, t)}}\, dt + \beta_m^{p-1}\frac{d}{d\beta_m}\z(\int_0^1 \frac{t^2}{\sqrt{\overline D(\beta_m, t)}}\, dt\y)\right]}_{>0} \\
& \cdot \underbrace{\frac{1}{\overline{{\bf T}}_{p}'(\beta_m)}}_{>0} \cdot \underbrace{\frac{1-T_*}{4(m-i)\sqrt{\lambda_m}}}_{>0} > 0.
		\end{split}
\]
		}

		To obtain the result, it is suffices to study  $\frac{d \underline{{\bf V}}_{p}}{d \alpha_m}$.
Using the product rule, we obtain that
\[
\frac{d\underline{{\bf V}}_{p}}{d\alpha_m} = (p-1)\alpha_m^{p-2} \int_0^1 \frac{t^2}{\sqrt{D(\alpha_m, t)}}\, dt - \alpha_m^{p-1} \frac{d}{d\alpha_m} \left( \int_0^1 \frac{t^2}{\sqrt{D(\alpha_m, t)}}\, dt \right).
\]
Thus
\[
\frac{d\underline{{\bf V}}_{p}}{d\alpha_m} = (p-1)\alpha_m^{p-2}\int_0^1 \frac{t^2}{\sqrt{D(\alpha_m, t)}}\, dt - \frac{\alpha_m^{p-1}}{2p^*} \int_0^1 \frac{t^2(1 - t^{2p^*})}{D(\alpha_m, t)^{3/2}}\, dt.
\]
We need to show
\[
(p-1)\int_0^1 \frac{t^2}{\sqrt{D(\alpha_m, t)}}\, dt - \frac{\alpha_m}{2p^*} \int_0^1 \frac{t^2(1 - t^{2p^*})}{D(\alpha_m, t)^{3/2}}\, dt > 0.
\]
It follows that
\[
\begin{split}
	&(p-1)\int_0^1 \frac{t^2}{\sqrt{D(\alpha_m, t)}}\, dt - \frac{\alpha_m}{2p^*} \int_0^1 \frac{t^2(1 - t^{2p^*})}{D(\alpha_m, t)^{3/2}}\\=&(p-1)\int_0^1\frac{t^2(1-t^2+\frac{\al_m}{p^*}(1-t^{2p^*}))dt}{D^{\frac{3}{2}}}-\frac{\al_m}{2p^*}\int_0^1\frac{t^2(1-t^{2p^*})}{D^{\frac{3}{2}}}dt\\
	=&(p-1)\int_0^1\frac{t^2(1-t^2)}{D^{\frac{3}{2}}}dt+\frac{2p-3}{2p^*}\al_m\int_0^1\frac{t^2(1-t^{2p^*})}{D^{\frac{3}{2}}}dt>0.
\end{split}
\]	
Then we obtain that $	G'(\lambda_m)<0$ and the solution of \eqref{lambda} is unique.

 {
The uniqueness of the eigenvalue $\lambda_m$ ensures that the solution $u_m(x)$ to the Schr\"odinger equation \eqref{P1} is uniquely determined. Furthermore, the one-to-one correspondence in \eqref{eq2} between $\hat q$ and $u_m$ guarantees the uniqueness of the potential $\hat q$.
}
\end{proof}


{
	\subsection{A concentrate example  ($p=2$)}\lb{sec5}
}

{For $p=2$, $m=2$ and $i=1$, we give more detailed results. Here the elliptic functions will play an important role.}

For $0<T_*<1/2$, we represent the six functions above through elliptic integration
	\[
    \begin{split}
	\underline{{\bf T}}_2(\alpha)&=\sqrt{\frac{2}{2+\alpha}}\mathbb E_1\left(-\frac{\alpha}{2+\alpha}\right),\\
    \underline {\bf U}_2(\alpha)&=\frac{4+2\alpha}{3}\left((4+\alpha)\mathbb E_1-4\mathbb E_2\left(-\frac{\alpha}{2+\alpha}\right)\right),\\
	\underline{{\bf V}}_2(\alpha)&=\sqrt{4+2\alpha}\left(-\mathbb E_1\left(-\frac{\alpha}{2+\alpha}\right)+\mathbb E_2\left(-\frac{\alpha}{2+\alpha}\right)\right),
	\\
	\overline{{\bf T}}_2(\beta)&=\sqrt{\frac{2}{2-\beta}}\mathbb E_1\left(\frac{\beta}{2-\beta}\right),\\
    \overline {\bf U}_2(\beta)& =\frac{\sqrt{4-2\beta}}{3}\left((4-\beta)\mathbb E_1\left(\frac{\beta}{2-\beta}\right)-4\mathbb E_2\left(\frac{\beta}{2-\beta}\right)\right),\\
    \overline{{\bf V}}_2(\beta)&=\sqrt{4-2\beta}\z(\mathbb E_1\left(\frac{\beta}{2-\beta}\right)-\mathbb E_2\left(\frac{\beta}{2-\beta}\right)\y),
\end{split}
\]
where $\alpha\in(0,+\infty)$, $\beta\in(0,1)$ and
\[
\begin{split}
 \mathbb E_1(s)&=\int_0^1\frac{1}{\sqrt{(1-t^2)(1-st^2)}}dt,
\\
\mathbb E_2(s)&=\int_0^1\frac{\sqrt{1-st^2}}{\sqrt{1-t^2}}dt.
\end{split}
\]
Then \eqref{0+T}, \eqref{1-T} and \eqref{third eq} become
\[
\begin{split}
T_*&=\frac{2}{\sqrt{\la_2}}\sqrt{\frac{2}{2+\al_2}}\mathbb E_1\left(-\frac{\al_2}{2+\al_2}\right),
\\
1-T_*&=\frac{2}{\sqrt{\la_2}}\sqrt{\frac{2}{2-\beta_2}}\mathbb E_1\left(\frac{\beta_2}{2-\beta_2}\right),
\end{split}
\]
and
\[
\begin{split}
&\sqrt{4+2\al_2}\left(-\mathbb E_1\left(-\frac{\al_2}{2+\al_2}\right)+\mathbb E_2\left(-\frac{\al_2}{2+\al_2}\right)\right)\\
=&\sqrt{4-2\beta_2}\left( \mathbb E_1\left(\frac{\beta_2}{2-\beta_2}\right)-\mathbb E_2\left(\frac{\beta_2}{2-\beta_2}\right)\right).
\end{split}
\]

For $\frac{1}{2}<T_*<1$, it follows that
\[
\begin{split}
\underline{{\bf T}}_2(\alpha)&=\sqrt{\frac{2}{2-\alpha}}\mathbb E_1\left(\frac{\alpha}{2-\alpha}\right),\\
\underline {\bf U}_2(\alpha)&=\frac{4-2\alpha}{3}\left((4-\alpha)\mathbb E_1-4\mathbb E_2\left(\frac{\alpha}{2-\alpha}\right)\right),
\\
\underline{{\bf V}}_2(\alpha)&=\sqrt{4-2\alpha}\left(-\mathbb E_1\left(\frac{\alpha}{2-\alpha}\right)+\mathbb E_2\left(\frac{\alpha}{2-\alpha}\right)\right),
\\
\overline{{\bf T}}_2(\beta)&=\sqrt{\frac{2}{2+\beta}}\mathbb E_1\left(-\frac{\beta}{2+\beta}\right),
\\
\overline {\bf U}_2(\beta)&=\frac{\sqrt{4+2\beta}}{3}\left((4+\beta)\mathbb E_1\left(-\frac{\beta}{2+\beta}\right)-4\mathbb E_2\left(-\frac{\beta}{2+\beta}\right)\right),
\\
\overline{{\bf V}}_2(\beta)&=\sqrt{4+2\beta}\z(\mathbb E_1\left(-\frac{\beta}{2+\beta}\right)-\mathbb E_2\left(-\frac{\beta}{2+\beta}\right)\y),
\end{split}
\]
where $\alpha\in(0,1)$, $\beta\in (0,+ \infty)$.\\
Then \eqref{0+T}, \eqref{1-T} and \eqref{third eq} become
\[ \begin{split}
T_*&=\frac{2}{\sqrt{\la_2}}\sqrt{\frac{2}{2-\al_2}}\mathbb E_1\left(\frac{\al_2}{2-\al_2}\right),
\\
1-T_*& =\frac{2}{\sqrt{\la_2}}\sqrt{\frac{2}{2+\beta_2}}\mathbb E_1\left(-\frac{\beta_2}{2+\beta_2}\right),
\end{split}
\]
and
\[
\begin{split}
\sqrt{4-2\al_2}\left(-\mathbb E_1\left(\frac{\al_2}{2-\al_2}\right)+\mathbb E_2\left(\frac{\al_2}{2-\al_2}\right)\right)
=\sqrt{4+2\beta_2} \left( \mathbb E_1\left(-\frac{\beta_2}{2+\beta_2}\right)-\mathbb E_2\left(-\frac{\beta_2}{2+\beta_2}\right)\right).
\end{split}
\]
For $0 < T_* < 1/2$,
note that
\[\underline{{\bf V}}_2(\alpha_2)=\sqrt{4+2\alpha_2}\left(-\mathbb E_1\left(-\frac{\alpha_2}{2+\alpha_2}\right)+\mathbb E_2\left(-\frac{\alpha_2}{2+\alpha_2}\right)\right).
\]
The uniqueness of solutions can be more rigorously demonstrated through elliptic function analysis. In fact,
Let 	\[
G_2(\lambda_2)=
\underline{{\bf V}}_{p}\left(\underline{{\bf T}}_{p}^{-1}\left(\frac{\sqrt{\la_2}T_*}{2}\right)\right)-\overline{{\bf V}}_{p}\left(\overline{{\bf T}}_{p}^{-1}\left(\frac{ \sqrt{\la_2}(1-T_*)}{2}\right)\right),\quad {\rm for}\ \lambda_2\in\mathbb R^+.\]
For $0 < T_* < 1/2$, 
note that
\[\underline{{\bf V}}_2(\alpha_2)=\sqrt{4+2\alpha_2}\left(-\mathbb E_1\left(-\frac{\alpha_2}{2+\alpha_2}\right)+\mathbb E_2\left(-\frac{\alpha_2}{2+\alpha_2}\right)\right).
\]
It follows that $\underline{{\bf V}}_2(\alpha_2)$ increases when $\alpha_2$ increases.
According to the proof of Theorem \ref{3.2}, we obtain that
 	 $G_2(\lambda_2)$ decreases when $\lambda_2 \in \mathbb{R}^+$ increases. It immediately follows that $\lambda_2$ is unique.
Similarly, for the case   $\frac{1}{2} < T_* < 1$, we can analogously  {prove} the uniqueness of $\lambda_2$,  from which it follows that
$\hat q$ is also unique.

\begin{figure}[ht]
\centering
\hspace{1.5cm}\includegraphics[width=8.5cm, height=6cm]{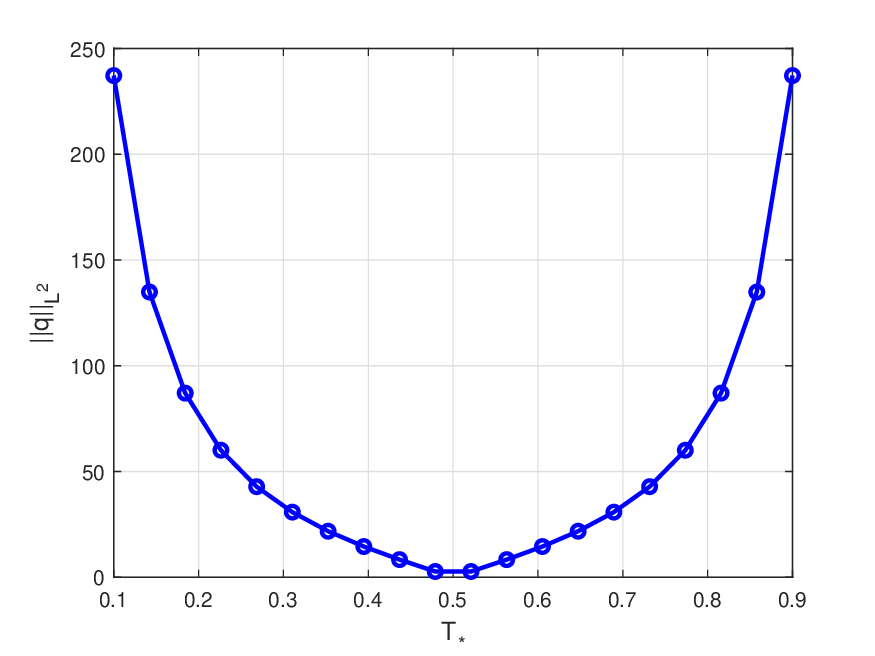}
\caption{The relationship between  $T_*$ and $\|\hat q\|_{\mathcal{L}^2}$.}
\lb{fig3}
\end{figure}

\begin{remark}
With the zero target, we reduce the inverse nodal problem of Sturm-Liouville operators to solving  equations \eqref{0+T}, \eqref{1-T} and \eqref{third eq}. Thus, the infinite-dimensional inverse problem simplifies to solving three nonlinear equations for three parameters. While obtaining explicit analytical solutions to these equations remains challenging, we can numerically investigate the  relationship between  $T_*$ and $\|\hat q\|_{\Lp}$. See Fig.  \ref{fig3} for the case $i=1$, $m=2$, $p=2$.
\end{remark}

\section{Further discussion and open problem}

This paper develops a method for addressing a novel and application-oriented inverse nodal problem in Sturm-Liouville operator. Unlike McLaughlin \cite{Mclaughlin},  Yang \cite{yang,yang2},  Guo and Wei \cite{guo-wei},  Wang and Yurko\cite{w-y}, Yang \cite{YANGCF}, who rely on {\em a dense subset} of the nodal set, we propose an optimization framework that recovers the potential $\hat{q}$ from {\em merely a finite set} of nodal data, achieving optimal approximation to the target potential $q_0$. The proposed approach capitalizes on the complete continuity of nodes with respect to potentials, thereby establishing the existence of solutions to the inverse problem. Using the Lagrange multiplier method, we reformulate the inverse nodal optimization problem as a solvability problem of a class of  Schr\"odinger equations. This reformulation enables systematic investigation for the inverse nodal problem, consequently, with constant targets, the infinite-dimensional inverse problem simplifies to solving three nonlinear equations for three parameters. Furthermore, the periodicity of the reconstructed potential $\hat{q}$ is shown to be completely determined by the periodicity of solutions to the corresponding Schr\"odinger equations. This approach reveals three novel characteristics that establish an explicit relationship between $\|\hat{q}\|_{L^p}$ and the crucial parameter $T_*$.

 {Given potentials in the Sturm-Liouville operators, Chu et al. \cite{jiedian} studied the direct nodal problem and
obtained the optimal characterizations of locations for all nodes. By analyzing the node minimization problem, they established the critical equations, which yield two equivalent formulations for characterizing the minimal nodes, expressed as nonlinear systems involving of $4$-dimensional or $2$-dimensional parameters. While they asserted the uniqueness of these nodes, no rigorous mathematical proof was provided for the uniqueness of solutions to the nonlinear equations in \cite{jiedian}. Crucially, such solutions inherently contain significant physical information, particularly the parameter $\lambda_m$ governing the system's intrinsic frequency-a key factor in quantum vibrational modes and energy structures. In this paper, we establish a sufficient condition ensuring the uniqueness of solution for the nonlinear system in the inverse problem, thereby advancing and refining the original framework of direct nodal problem in Chu et al. \cite{jiedian}.}

 Of importance, we proved the  uniqueness of the potential $\hat q$ under the assumption $p>\frac{3}{2}$. For $1<p\leq \frac{3}{2}$, we guess that the uniqueness still holds. However, the proof in this case remains open, because we cannot find an efficient method to prove it. We leave it as an open problem.

 Finally, in this paper, for {\em one} given node, we present a framework to  study the inverse nodal  problems in  Sturm-Liouville operators. It is believed that it could be extended to the cases of given {\em two} or {\em finite } nodes. However, there will be a challenge of identifying the uniqueness of the potential $\hat q$.

\section*{ Conflict of Interest}
The authors declare that they have no conflict of interest.

\section*{Data Availability Statement}

No data was used in the research in this manuscript.

\section*{Acknowledgement}
\hskip\parindent
\small
We
declare that the authors are ranked in alphabetic order of their names and all of them have the same
contributions to this paper.

\end{document}